\documentclass[a4paper,12pt]{article}

\usepackage[left=2cm,right=2cm, top=2cm,bottom=3cm,bindingoffset=0cm]{geometry}

\usepackage{verbatim}
\usepackage{amsmath}
\usepackage{amsthm}
\usepackage{amssymb}
\usepackage{delarray}
\usepackage{cite}
\usepackage{hyperref}
\usepackage{mathrsfs}
\usepackage{tikz}
\usetikzlibrary{patterns}
\usepackage{caption}
\DeclareCaptionLabelSeparator{dot}{. }
\newcommand{\al}{\alpha}
\newcommand{\be}{\beta}

\newcommand{\de}{\delta}
\newcommand{\la}{\lambda}
\newcommand{\om}{\omega}

\newcommand{\eps}{\varepsilon}

\theoremstyle{plain}

\numberwithin{equation}{section}

\newtheorem{thm}{Theorem}[section]
\newtheorem{lem}[thm]{Lemma}

\newtheorem{cor}[thm]{Corollary}

\theoremstyle{definition}

\newtheorem{ip}[thm]{Inverse Problem}

\newtheorem{defin}[thm]{Definition}

\theoremstyle{remark}

\newtheorem{remark}[thm]{Remark}

 \allowdisplaybreaks

\begin{document}

\begin{center}
{\Large\bf Third-order differential operators\\[0.2cm] with a second-order distribution coefficient}
\\[0.5cm]
{\bf Natalia P. Bondarenko}
\end{center}

\vspace{0.5cm}

{\bf Abstract.} In this paper, we study differential operators associated with the formal expression $y''' + s(\sigma' y)' + s \sigma' y' + \kappa \sigma'' y$ with distribution coefficient $\sigma'' \in W_3^{-2}$, where $s$ and $\kappa$ are constants. The uniqueness theorems are proved for the inverse spectral problems that consist in the recovery of $\sigma$ from the Weyl-Yurko matrix on a finite interval and on the half-line. In addition, we discuss the reconstruction of $\sigma$ and formulate some open problems.

\medskip


\medskip


\section{Introduction} \label{sec:intr}

In this paper, we study differential operators associated with the formal expression 
\begin{equation} \label{expr}
\ell_{\sigma, s, \kappa}(y) := y''' + s(\sigma' y)' + s \sigma' y' + \kappa \sigma'' y, \quad s \in \{ 0, 1 \}, \quad \kappa \in \mathbb C,
\end{equation}
where $\sigma$ is a complex-valued function of $L_3(0,T)$, $T \le \infty$, and $s + |\kappa| \ne 0$.

Mirzoev and Shkalikov have proposed a regularization approach for differential operators with distribution coefficients for even orders in \cite{MS16} and for odd orders in \cite{MS19}. This approach allows one to introduce quasi-derivatives and to reduce 
differential equations with coefficients from certain generalized function spaces to first-order systems with integrable coefficients. 

The regularization of Mirzoev and Shkalikov was applied to studying asymptotics of solutions \cite{KM19, SS20, MK20, KMS23} and inverse spectral problems \cite{Bond21, Bond22, Bond23-mmas, Bond23-res, Bond23-reg, Bond24}. Some ideas of regularizing differential operators with distribution coefficients were developed by
Vladimirov \cite{Vlad04, Vlad17} and then applied to the investigation of oscillation properties for fourth-order operators in \cite{VS21}. In \cite{But24}, Buterin introduced a nonlocal generalization of the quasi-derivatives from \cite{MS16, MS19}. In \cite{But25}, he proposed an alternative regularization approach, which consists in the reduction to a special operator-differential expression, and proved the completeness of root functions for a wide class of operators including the higher-order ones with distribution coefficients. However, the differential expression \eqref{expr} has coefficients of higher singularity order than the Mirzoev-Shkalikov construction \cite{MS16, MS19} allows to handle. Specifically, the odd order regularization of \cite{MS19} works for $\ell_{\sigma,s,\kappa}(y)$ only with $\sigma \in W_{1,loc}^1(0,T)$.

In this paper, we present a matrix function associated with the differential expression \eqref{expr} and study the corresponding first-order system. Furthermore, we investigate inverse spectral problems that consist in the recovery of $\sigma$ together with some coefficients of boundary conditions from the so-called Weyl-Yurko matrix. This spectral characteristic generalizes the Weyl matrix that was introduced by Yurko \cite{Yur89, Yur92, Yur02} for higher-order differential operators with regular (integrable) coefficients. Our approach is based on the method of spectral mappings, which has been developed in \cite{Yur89, Yur92, Yur02, Leib72} for the case of regular coefficients and in \cite{Bond21, Bond22, Bond23-mmas, Bond23-res, Bond23-reg, Bond24} for differential equations with distribution coefficients admitting Mirzoev-Shkalikov regularization. As a result, we prove the uniqueness theorems for recovering $\sigma$ and coefficients of the boundary conditions on a finite interval and on the half-line. In addition, we discuss the reconstruction of $\sigma$ from the Weyl-Yurko matrix and formulate a series of open problems. 

We denote by $\de_{k,j}$ the Kronecker delta, by $I$ the unit matrix of size $(3 \times 3)$, by the upper index $^T$ the transpose of vectors, and by $C_0^{\infty}(0,T)$ the space of infinitely differentiable functions of finite support in $(0,T)$.

\section{Associated matrix}

Denote by $\mathfrak F_n$ ($n \ge 2$) the class of $(n \times n)$ matrix functions $F(x) = [f_{kj}(x)]_{k,j=1}^n$ satisfying the conditions
\begin{equation} \label{condF}
f_{kj} = \de_{k+1,j}, \quad k < j, \qquad f_{kj} \in L_{1,loc}(0,T), \quad k \ge j.
\end{equation}

For every $F \in \mathfrak F_n$, one can define the quasi-derivatives
\begin{equation} \label{quasi}
y^{[0]} := y, \quad y^{[k]} := (y^{[k-1]})' - \sum_{j = 1}^k f_{kj} y^{[j-1]}, \quad k = \overline{1,n}.
\end{equation}

Introduce the notion of associated matrix analogously to \cite[Definition~1]{MS16}.

\begin{defin} \label{def:ass}
A matrix function $F(x; \sigma) = [f_{kj}(x; \sigma)]_{j,k=1}^n$ is called an \textit{associated matrix} for a formal differential expression $\ell_{\sigma}$ of order $n$ with a coefficient $\sigma$ of some class $\mathcal K(0,T)$
if the following conditions are satisfied:

(i) $f_{kj}(x; \sigma) = \phi_{kj}(\sigma(x))$,
where $\phi_{kj}$ are infinitely differentiable functions; 

(ii) $F(\cdot; \sigma) \in \mathfrak F_n$ for each fixed $\sigma \in \mathcal K(0,T)$;

(iii) $\ell_{\sigma}(y) \equiv y^{[n]}$
for all $\sigma$ and $y$ in $C^{\infty}_0(0,T)$, where the quasi-derivative is defined by \eqref{quasi} using the elements $f_{kj}(x;\sigma)$.
\end{defin}

If a matrix function $F(x) = F(x;\sigma)$ is associated with some differential expression $\ell_{\sigma}$, then, for sufficiently smooth $\sigma$, the equation 
\begin{equation} \label{eqvl}
\ell_{\sigma}(y) = \lambda y, \quad x \in (0,T), 
\end{equation}
is equivalent to the first order system
\begin{equation} \label{sys}
Y' = (F(x) + \Lambda) Y, \quad x \in (0,T),
\end{equation}
where $Y = [y^{[0]}, y^{[1]}, \dots, y^{[n-1]}]^T$ and $\Lambda$ is the constant $(n \times n)$ matrix, whose element at position $(n,1)$ equals $\la$ and all the other elements equal zero.
Indeed, the first $(n-1)$ rows of \eqref{sys} are equivalent to \eqref{quasi} for $k = \overline{1,n-1}$ and the $n$-th row, to the equation $y^{[n]} = \lambda y$. On the other hand, in contrast to equation \eqref{eqvl}, the system \eqref{sys} is correctly defined for every $\sigma \in \mathcal K(0,T)$, since $f_{kj} \in L_{1,loc}(0,T)$ due to condition~(ii) of Definition~\ref{def:ass}.

Direct calculations show that the matrix function 
\begin{equation} \label{defF}
F = \begin{bmatrix}
        -(s+\kappa) \sigma & 1 & 0 \\
        -\bigl(\tfrac{3}{2}\kappa^2 + \tfrac{1}{2}s^2 + 2\kappa s\bigr)\sigma^2 & 2 \kappa s \sigma & 1 \\
        \kappa(\kappa^2 - s^2)\sigma^3 & -\bigl(\tfrac{3}{2}\kappa^2 + \tfrac{1}{2}s^2 - 2\kappa s\bigr)\sigma^2 & (s-\kappa)\sigma
    \end{bmatrix}
\end{equation}
is associated with the differential expression \eqref{expr} for $\sigma \in L_{3,loc}(0,T)$ and, in the case $s = 1$ and $\kappa \in \{-1, 0, 1\}$, for $\sigma \in L_{2,loc}(0,T)$.
Moreover, the relation $\ell_{\sigma,s,\kappa}(y) = y^{[3]}$ holds for every $\sigma \in C^2$ and $y \in C^3$. This motivates the investigation of the system \eqref{sys} with the matrix \eqref{defF} for $\sigma \in L_3(0,T)$. We mention that the matrix \eqref{defF} for $s = 0$, $\kappa = -1$ appeared in \cite{VNS21}.

\begin{remark}
For $y \in W_{3/2,loc}^2(0,T)$, the expression $\ell_{\sigma,s,\kappa}(y)$ can be understood as a generalized function (a linear continuous functional) acting on a test function $z \in C_0^{\infty}(0,T)$ as follows:
$$
(y''' + s(\sigma' y)' + s \sigma'y' + \kappa\sigma'' y, z) = \int_0^T \bigl( -y'' z' + (\kappa+s)\sigma y z'' + 2\kappa\sigma y'z' + (\kappa-s)\sigma y'' z \bigr) \, dx.
$$
However, this meaning is not related to the system \eqref{sys}, because the requirement $y^{[1]} = y' + (s + \kappa)\sigma y \in AC_{loc}(0,T)$, which is essential for the correctness of \eqref{sys}, implies that $y \not\in W_{3/2,loc}^2(0,T)$ if $(s + \kappa)\sigma \not\in AC_{loc}(0,T)$. Thus, the matrix \eqref{defF} is \textit{associated} with the differential expression \eqref{expr} but not \textit{compatible} in the sense of \cite[Definition~2]{MS16}.
Note that a family of matrices that are compatible but not associated with some differential expressions was constructed in \cite{Bond23-reg}.
\end{remark}

\section{Inverse problem on a finite interval} \label{sec:fin}

Consider the system \eqref{sys} with the matrix \eqref{defF} for $\sigma \in L_3(0,1)$ and fixed numbers $s$, $\kappa$.

Let us introduce the boundary conditions analogously to the inverse problem theory for higher-order differential operators \cite{Yur02}.
Define $(3 \times 3)$ matrices $U = P_U L_U$ and $V = P_V L_V$, where $P_U$ and $P_V$ are permutation matrices, $L_U$ and $L_V$ are unit lower triangular matrices. Denote by $U_k$ and $V_k$ ($k = \overline{1,3}$) the rows of the matrices $U$ and $V$, respectively. Boundary conditions will be defined by the linear forms
$$
U_i Y(0) = \sum_{j = 1}^{p_i} u_{ij} y_j(0), \quad V_l Y(1) = \sum_{j = 1}^{q_l} v_{lj} y_j(1),
$$
where $(p_1, p_2, p_3)$ and $(q_1, q_2, q_3)$ are the permutations corresponding to the matrices $P_U$ and $P_V$, respectively, $U = [u_{ij}]_{i,j=1}^3$, $V = [v_{lj}]_{l,j=1}^3$, $u_{i,p_i} = 1$, $v_{l,q_l} = 1$, $Y(x) = [y_j(x)]_{j = \overline{1,3}}^T$.

Let $C(x, \la)$ be the $(3 \times 3)$-matrix solution of the initial value problem for the system \eqref{sys} with the condition 
\begin{equation} \label{initC}
C(0, \la) = U^{-1}. 
\end{equation}

Next, introduce the vector solutions $\Phi_k(x, \la)$ ($k = \overline{1,3}$) of the system \eqref{sys} satisfying the boundary conditions
\begin{equation} \label{bcPhi}
U_i \Phi_k(0,\la) = \de_{i,k}, \quad i = \overline{1,k}, \qquad
V_l \Phi_k(1, \la) = 0, \quad l = \overline{1,3-k}.
\end{equation}

The elements of the matrix function $C(x, \la)$ are entire functions in $\la$ of order not greater than $\frac{1}{3}$ for each fixed $x \in [0,1]$, and $\Phi_3(x, \la)$ coincides with the third column of $C(x, \la)$. For each $k \in \{ 1, 2 \}$ and fixed $x \in [0,1]$, the elements of $\Phi_k(x, \la)$ are meromorphic in $\la$ and have poles at the eigenvalues of the boundary value problem for the system \eqref{sys} with the conditions
\begin{equation} \label{bcY} 
U_i Y(0) = 0, \quad i = \overline{1,k}, \qquad 
V_l Y(1) = 0, \quad l = \overline{1, 3-k}.
\end{equation}

Denote by $\Phi(x, \la)$ the $(3 \times 3)$ matrix of the columns $\Phi_k(x, \la)$ ($k = \overline{1,3}$). Obviously, the columns of the matrices $C(x, \la)$ and $\Phi(x, \la)$ form fundamental systems of solutions for \eqref{sys}. Therefore, they are related as follows: \begin{equation} \label{defM}
\Phi(x, \la) = C(x,\la) M(\la), 
\end{equation}
where $M(\la)$ is the so-called \textit{Weyl-Yurko matrix}. Consider the following inverse spectral problem.

\begin{ip} \label{ip:finite}
Suppose that the matrices $P_U$ and $P_V$ are known.
Given the Weyl-Yurko matrix $M(\la)$, find $\sigma$ and the elements of the matrices $U$ and $V$ whenever possible.
\end{ip}

For simplicity, we assume that the structure of the boundary conditions, which is specified by the permutation matrices $P_U$ and $P_V$, is fixed. In fact $P_U$ can be found from the asymptotics of the Weyl-Yurko matrix, which is analogous to (2.1.19) in \cite{Yur02}.
Since the elements of the Weyl-Yurko matrix $M(\la)$ are meromorphic, in Theorem~\ref{thm:uniqf} it is sufficient to specify them in a countable set of regular points $\{ z_k \}_{k \ge 1}$ having a limit point.   

First, let us discuss the recovery of the matrix $V$.
Note that the vector $V_3$ does not participate in the problem statement, so it cannot be determined from $M(\la)$.

\begin{lem} \label{lem:findV}
Given the permutation $(q_1, q_2, q_3)$ and the matrix function $\Phi(x, \la)$, one can uniquely find the vector $V_1$ and the following information about $V_2$ in the possible cases:
\begin{itemize}
\item $q_2 = 1$: $V_2 = [1, 0, 0]$.
\item $(q_1, q_2) = (1, 2)$: $V_1 = [1, 0, 0]$, $V_2 = [v_{21}, 1, 0]$, where $v_{21}$ can be arbitrary.
\item $(q_1, q_2) = (1, 3)$: $V_1 = [1, 0, 0]$, $V_2 = [v_{21}, v_{22}, 1]$, where $v_{22}$ is uniquely determined and $v_{21}$ can be arbitrary.
\item $(q_1, q_2) = (2, 3)$: $V_1 = [v_{11}, 1, 0]$, $V_2 = [v_{21}, v_{22}, 1]$, where the number $(v_{21} - v_{11} v_{22})$ is uniquely determined.
\item $(q_1, q_2) = (3, 2)$: $V_1 = [v_{11}, v_{12}, 1]$, $V_2 = [v_{21}, 1, 0]$, where $v_{21}$ is uniquely determined.
\end{itemize}
\end{lem}

\begin{proof}
Consider the case $(q_1, q_2) = (1, 3)$. Then the vector function $\Phi_1(x,\la) = [\Phi_{j1}(x,\la)]_{j = \overline{1,3}}^T$ satisfies the conditions
$$
\Phi_{11}(1,\la) = 0, \quad v_{21} \Phi_{11}(1,\la) + v_{22} \Phi_{21}(1,\la) + \Phi_{31}(1,\la) = 0,
$$
which are equivalent for each $v_{21}$. We find $v_{22} = -\dfrac{\Phi_{31}(1,\la)}{\Phi_{21}(1,\la)}$, where the denominator is not identically zero. Otherwise, one gets $\Phi_{j1}(1,\la) \equiv 0$ for $j = \overline{1,3}$, so $\Phi_1(x,\la) \equiv 0$, which is impossible due to $U_1 \Phi_1(0,\la) = 1$. The vector functions $\Phi_2(x, \la)$ and $\Phi_3(x, \la)$ do not give us additional information on $V_2$. The other cases can be considered analogously.
\end{proof}

Let us call the triple $(\sigma, U, V)$ by the problem $\mathcal L$. Along with $\mathcal L$, consider another problem $\tilde{\mathcal L} = (\tilde \sigma, \tilde U, \tilde V)$ of the same form but with different coefficients. We agree that, if a symbol $\al$ denotes an object related to $\mathcal L$, then the symbol $\tilde \al$ will denote the similar object related to $\tilde{\mathcal L}$.

\begin{thm} \label{thm:uniqf}
Suppose that $U = \tilde U$, $P_V = \tilde P_V$, and $M(\la) \equiv \tilde M(\la)$. Then $\sigma(x) = \tilde \sigma(x)$ a.e. on $(0,1)$. In other words, the Weyl-Yurko matrix $M(\la)$ uniquely specifies the function $\sigma$ in $L_3(0,1)$, while $U$ and $P_V$ are known. Moreover, the vector $V_1$ and the certain elements of the vector $V_2$ are uniquely specified by $M(\la)$ according to Lemma~\ref{lem:findV}.
\end{thm}

\begin{proof}
Define the matrix of spectral mappings
$$
R(x, \la) = [r_{jk}(x,\la)]_{j,k = 1}^3 := \Phi(x, \la) \tilde \Phi^{-1}(x,\la).
$$
Using \eqref{defM} and the equality $M(\la) \equiv \tilde M(\la)$, we get
\begin{equation} \label{RC}
    R(x, \la) = C(x, \la) \tilde C^{-1}(x, \la).
\end{equation}

One can easily show that $\det(\tilde C(x, \la)) \equiv \pm 1$. Consequently, the elements of $\tilde C^{-1}(x, \la)$ are entire functions in $\la$ of order not greater than $\frac{1}{3}$ for each fixed $x \in [0,1]$, and so do the elements $r_{jk}(x, \la)$. On the other hand, they possess the following asymptotics (see \cite[p.~14]{Bond21}):
\begin{equation} \label{asymptr}
r_{jk}(x, \la) = \begin{cases}
o(1), & j < k, \\
1 + o(1), & j = k, \\
o(\la), & j > k,
\end{cases}
\qquad |\la| \to \infty, \quad \arg \la = \be \not\in \{0, \pi\}, \quad x \in [0,1).
\end{equation}
Indeed, the proof of the asymptotics \eqref{asymptr} in \cite{Bond21} relies on the conditions \eqref{condF} for the associated matrix $F(x)$ but does not use its specific construction. Applying Phragmen-Lindel\"of's Theorem (see \cite[Theorem~21]{Lev56}) and Liouville's Theorem, we conclude that $R(x, \la)$ is a constant unit triangular matrix $R(x)$ for each $x \in [0,1)$. Thus
\begin{equation} \label{Rx}
R(x) \tilde \Phi(x, \la) = \Phi(x, \la).
\end{equation}

Differentiating \eqref{Rx} and using the relations
\begin{equation} \label{eqPhi}
\Phi'(x, \la) = (F(x) + \Lambda) \Phi(x, \la), \quad
\tilde \Phi'(x, \la) = (\tilde F(x) + \Lambda) \tilde \Phi(x, \la),   
\end{equation}
we derive
\begin{equation} \label{eqR}
R'(x) + R(x) \tilde F(x) = F(x) R(x), \quad x \in (0,1).
\end{equation}

Rewrite \eqref{eqR} in the element-wise form:
$$
\begin{bmatrix}
0 & 0 & 0 \\
r_{21}' & 0 & 0 \\
r_{31}' & r_{32}' & 0
\end{bmatrix} +
\begin{bmatrix}
1 & 0 & 0 \\
r_{21} & 1 & 0 \\
r_{31} & r_{32} & 1
\end{bmatrix}
\begin{bmatrix}
a_{11} \tilde \sigma & 1 & 0 \\
a_{21} \tilde \sigma^2 & a_{22} \tilde \sigma & 1 \\
a_{31}\tilde \sigma^3 & a_{32} \tilde \sigma^2 & a_{33}\tilde \sigma
\end{bmatrix} = 
\begin{bmatrix}
a_{11} \sigma & 1 & 0 \\
a_{21} \sigma^2 & a_{22} \sigma & 1 \\
a_{31}\sigma^3 & a_{32} \sigma^2 & a_{33}\sigma
\end{bmatrix}  
\begin{bmatrix}
1 & 0 & 0 \\
r_{21} & 1 & 0 \\
r_{31} & r_{32} & 1
\end{bmatrix},
$$
where $a_{kj}$ are the constants from \eqref{defF}: $f_{kj} = a_{kj}\sigma^{k-j+1}$.

From this system, we find
\begin{align} \label{r21}
& r_{21} = -a_{11} \hat \sigma, \quad r_{32} = a_{33} \hat \sigma, \quad \hat \sigma := \sigma - \tilde \sigma, \\ \label{r21p}
& r_{21}' + a_{11} r_{21} \tilde \sigma + a_{21} \tilde \sigma^2 = a_{21} \sigma^2 + a_{22} \sigma r_{21} + r_{31}, \\ \label{r32p}
& r_{32}' + r_{31} + a_{22} r_{32} \tilde \sigma + a_{32} \tilde \sigma^2 = a_{32} \sigma^2 + a_{33} r_{32} \sigma, \\
\label{r31p}
& r_{31}' + a_{11} r_{31} \tilde \sigma + a_{21} r_{32} \tilde \sigma^2 + a_{31} \tilde \sigma^3 = a_{31}\sigma^3 + a_{32} \sigma^2 r_{21} + a_{33} \sigma r_{31}.
\end{align}

According to \eqref{eqPhi}, the matrix functions $\Phi(x, \la)$ and $\tilde \Phi(x, \la)$ are absolutely continuous by $x \in [0,1]$, and so do $r_{jk}(x)$. Hence \eqref{r21} implies $\hat \sigma \in AC[0,1]$. Substituting \eqref{r21} into \eqref{r21p} and \eqref{r32p}, we obtain
\begin{align*} 
r_{31} & = -a_{11} \hat \sigma' + \bigl( (a_{22}a_{11} - a_{21})\sigma - (a_{11}^2 + a_{21})\tilde \sigma\bigr) \hat \sigma \\
& = -a_{33} \hat \sigma' + \bigl( (a_{32} - a_{22}a_{33}) \tilde \sigma + (a_{32} + a_{33}^2)\sigma \bigr) \hat \sigma.
\end{align*}
Using the formulas for the coefficients \eqref{defF}, we get
\begin{equation} \label{relr31}
r_{31} = (s+\kappa) \hat \sigma' + \frac{1}{2}(s^2 - \kappa^2) \hat \sigma^2 = -(s-\kappa) \hat \sigma' + \frac{1}{2}(s^2 - \kappa^2) \hat \sigma^2.
\end{equation}

In the case $s = 1$, the equality \eqref{relr31} implies $\hat \sigma' = 0$.
In the case $s = 0$, wlog $\kappa = 1$ and $r_{31} = \hat \sigma' - \frac{1}{2} \hat \sigma^2$. 
Consequently, we have $\hat \sigma' \in AC[0,1]$. Using \eqref{r21} and \eqref{r31p}, we obtain $\hat \sigma'' = 0$.

It follows from \eqref{RC} and $U = \tilde U$ that $R(0) = I$. In particular $r_{21}(0) = r_{31}(0) = r_{32}(0) = 0$, so $\hat \sigma(0) = \hat \sigma'(0) = 0$. Therefore, in the both cases $s = 0, 1$, we get $\hat \sigma \equiv 0$, which implies $r_{21}(x) \equiv r_{32}(x) \equiv r_{31}(x) \equiv 0$. Thus $\sigma(x) = \tilde \sigma(x)$ a.e. on $(0,1)$, $R(x) \equiv I$ and $\Phi(x, \la) \equiv \tilde \Phi(x, \la)$. Applying Lemma~\ref{lem:findV} concludes the proof.
\end{proof}

\begin{remark}
It follows from the proof of Theorem~\ref{thm:uniqf} that, if the matrix $U$ were unknown, then the Weyl-Yurko matrix $M(\la)$ would uniquely specify the distribution $\sigma'$ for $s = 1$ and $\sigma''$ for $s = 0$ but not $\sigma$.   
\end{remark}

For the case $s = 1$, Theorem~\ref{thm:uniqf} can be strengthened.
The following corollary shows that
the specification of a single element of the matrix $L_U = [l_{jk}]_{j,k = 1}^3$ is sufficient to uniquely determine the other two.

\begin{cor} \label{cor:U}
Suppose that $s = 1$, $P_U = \tilde P_U$, $P_V = \tilde P_V$, and $M(\la) \equiv \tilde M(\la)$, and one of the following conditions is fulfilled:
\begin{itemize}
\item $\kappa = 1$ and $l_{21} = \tilde l_{21}$.
\item $\kappa = -1$ and $l_{32} = \tilde l_{32}$.
\item $\kappa \ne \pm 1$ and $l_{jk} = \tilde l_{jk}$ for at least one pair of indices $(j,k) \in \{ (2,1), (3,1), (3,2) \}$.
\end{itemize}
Then $\sigma(x) = \tilde \sigma(x)$ a.e. on $(0,1)$ and $U = \tilde U$.
\end{cor}

\begin{proof}
On the one hand, the proof of Theorem~\ref{thm:uniqf} for $s = 1$ implies $\hat \sigma' = 0$, so $\hat\sigma = c$, where $c$ is a constant. Therefore, it follows from \eqref{r21} and $\eqref{relr31}$ that
\begin{equation} \label{sm1}
r_{21} = -(1+\kappa) c, \quad r_{32} = -(1-\kappa) c, \quad r_{31} = \tfrac{1}{2} (1 - \kappa^2) c.
\end{equation}

On the other hand, the relations \eqref{RC}, \eqref{initC}, and $P_U = \tilde P_U$ imply $R(0) = L_U^{-1} \tilde L_U$, so
\begin{equation} \label{sm2}
r_{21}(0) = \tilde l_{21} - l_{21}, \quad
r_{32}(0) = \tilde l_{32} - l_{32}, \quad
r_{31}(0) = l_{32}(l_{21} - \tilde l_{21}) + \tilde l_{31} - l_{31}.
\end{equation}

Consider the case $\kappa \ne \pm 1$ and $l_{31} = \tilde l_{31}$. Then \eqref{sm1} and \eqref{sm2} imply
$$
r_{31}(0) = l_{32}(l_{21} - \tilde l_{21}) = l_{32} (1 + \kappa) c = \tfrac{1}{2}(1 - \kappa^2)c.
$$
Hence $c = 0$ or $l_{32} = \tfrac{1}{2} (1 - \kappa) c$. In the latter case, from the expressions of $r_{32}$, we get $\tilde l_{32} = \tfrac{1}{2} (1 - \kappa) c$, so $r_{32}(0) = 0$ and $c = 0$. Anyway, we conclude that $\hat \sigma \equiv 0$, $l_{21} = \tilde l_{21}$ and $l_{32} = \tilde l_{32}$ from \eqref{sm1} and \eqref{sm2}.
The other cases are studied analogously.
\end{proof}

\begin{remark}
As in \cite{Bond21}, the elements of the Weyl-Yurko matrix can be represented in the form
$$
m_{jk}(\la) = -\frac{\Delta_{jk}(\la)}{\Delta_{kk}(\la)}, \quad 1 \le k < j \le 3,
$$
where $\Delta_{jk}(\la)$ ($1 \le k \le j \le 3$) are the analytic characteristic functions, which can be uniquely recovered from their zeros, coinciding with the eigenvalues of the corresponding problems $\mathcal L_{jk}$ for the system \eqref{sys} with the boundary conditions
$$
U_i Y(0) = 0, \quad i = \overline{1,k-1}, j, \qquad
V_l Y(1) = 0, \quad l = \overline{1,3-k}.
$$
Consequently, the Weyl-Yurko matrix $M(\la)$ is uniquely determined by specification of the five spectra of the problems $\mathcal L_{11}$, $\mathcal L_{21}$, $\mathcal L_{22}$, $\mathcal L_{31}$, $\mathcal L_{32}$, and so is $\sigma$. Therefore, Theorem~\ref{thm:uniqf} can be considered as an analog of the famous result by Borg \cite{Borg46}, which says that the potential of the second-order Sturm-Liouville operator is uniquely specified by two spectra.
\end{remark}

\section{Inverse problem on the half-line} \label{sec:half}

In this section, we suppose that $\sigma \in (L_1 \cap L_3)(\mathbb R_+)$, $\mathbb R_+ := (0,+\infty)$, $s$ and $\kappa$ are fixed.
Then the elements of the matrix $F(x)$ defined by \eqref{defF} belong to $L_1(\mathbb R_+)$.

Define the constant matrix $U = P_U L_U$ and the matrix solution $C(x,\la)$ of the system~\ref{sys} as in Section~\ref{sec:fin}. Let $\la = \rho^n$. For each fixed value of $\arg \rho$, suppose that the roots $\{ \om_k \}_{k = 1}^n$ of the equation $\om^n = 1$ are numbered so that
$$
\mbox{Re}(\rho \om_1) \le \mbox{Re}(\rho \om_2) \le \dots \le \mbox{Re}(\rho \om_n).
$$
Denote by $\Phi_k(x,\la)$ ($k = \overline{1,3}$) the solutions of \eqref{sys} satisfying the boundary conditions
$$
U_i \Phi_k = \de_{i,k}, \quad i = \overline{1,k}, \quad \Phi_k(x,\la) = O(\exp(\rho \om_k x)), \quad x \to \infty.
$$
Then $\Phi(x,\la) = C(x,\la) M(\la)$, where $M(\la) = [m_{jk}(\la)]_{j,k = 1}^3$ is the Weyl-Yurko matrix for the half-line case. The elements $m_{jk}(\la)$ ($j > k$) are analytic in the $\la$-plane with the cut $(-1)^{k+1} \la \ge 0$ except for at most countable bounded set of poles as in \cite[Theorem~2.1.1]{Yur02}.

Consider the inverse spectral problem.

\begin{ip} \label{ip:half}
Given the Weyl-Yurko matrix $M(\la)$, find $\sigma$ and $U$.
\end{ip}

We prove the uniqueness of solution for Inverse Problem~\ref{ip:half}:

\begin{thm} \label{thm:uniqh}
If $M(\la) \equiv \tilde M(\la)$, then $\sigma(x) = \tilde \sigma(x)$ a.e. on $\mathbb R_+$ and $U = \tilde U$.
\end{thm}

\begin{proof}
The elements of the Weyl-Yurko matrix possess asymptotics analogous to (2.1.11) in \cite{Yur02}:
$$
m_{jk}(\rho^n) = \rho^{p_j - p_k} \mu_{mk} (1 + o(1)), \quad |\rho| \to \infty,   
$$
where $(p_1, p_2, p_3)$ is the permutation related to $P_U$ and $\mu_{mk} \ne 0$ are some constants. Consequently, the matrix $P_U$ is uniquely determined by the Weyl-Yurko matrix.

Repeating the arguments from the proof of Theorem~\ref{thm:uniqf}, we conclude that $\hat \sigma'' = 0$. Hence $\hat \sigma(x) = c_1 x + c_2$. Since $\sigma$ and $\tilde \sigma$ belong to $L_1(\mathbb R_+)$, so does $\hat \sigma$, then $\hat \sigma \equiv 0$. Thus $\sigma(x) = \tilde \sigma(x)$ a.e. on $\mathbb R_+$. 

Using \eqref{r21} and \eqref{relr31}, we get $R(0) = I$. It follows from \eqref{RC} and \eqref{initC} that $R(0) = U^{-1} \tilde U$, so $U = \tilde U$.
\end{proof}

\begin{remark}
In contrast to the finite interval case, the Weyl-Yurko matrix $M(\la)$ for the half-line uniquely specifies the matrix $U$, roughly speaking, because we have the additional boundary condition $\sigma(\infty) = 0$. This result is analogous to \cite[Theorem~4.3]{Bond23-mmas}, which is valid for $\sigma \in W_1^1(\mathbb R_+)$. For a finite interval, Theorem~5.3 in \cite{Bond23-mmas} implies the uniqueness of recovering the elements $l_{21}$ and $l_{32}$ of the matrix $L_U$ for $\sigma \in W_1^1[0,1]$ when $l_{31}$ is known. However, our results for $\sigma \in L_3(0,1)$ differ (see Theorem~\ref{thm:uniqf} and Corollary~\ref{cor:U}).
In particular, for $s = 0$, elements of $L_U$ are not determined by the Weyl-Yurko matrix.
\end{remark}

\begin{remark}
For $s = 1$ and $\kappa \in \{-1, 0, 1 \}$, the results of Sections~\ref{sec:fin} and~\ref{sec:half} are valid for $\sigma$ in $L_2(0,1)$ and $(L_1 \cap L_2)(\mathbb R_+)$, respectively.
\end{remark}

\section{Reconstruction and open problems}

The solutions of the inverse spectral problems for the system \eqref{sys} can be constructed by the method of spectral mappings \cite{Yur02, Bond22, Bond24}. In particular, Theorem~1 and Algorithm~1 in \cite{Bond22} and Theorem~1 in \cite{Bond24} are applicable, because they are valid for any matrix $F \in \mathfrak F_n$ satisfying the additional conditions $f_{kk} \in L_2(0,1)$, $f_{kj} \in L_1(0,1)$, and $\mbox{trace}\,(F) = 0$, which hold for the matrix \eqref{defF}. Those results allow one to derive a linear main equation for Inverse Problem~\ref{ip:finite} and to recover the Weyl solutions $\Phi_k(x,\la)$ from the Weyl-Yurko matrix. Furthermore, the main equation is uniquely solvable by necessity due to \cite[Theorem~1]{Bond22} and, in the self-adjoint case under some additional restrictions on the spectral data, by sufficiency due to \cite[Theorem~1]{Bond24}. By self-adjoint, we mean the case of real-valued $\sigma(x)$ and $\kappa \in \{ 0, \mathrm{i} \}$. Then the differential expression $\mathrm{i} \ell_{\sigma}(y)$ is formally self-adjoint. However, the reconstruction of distribution coefficients of a differential expression from the solutions $\Phi_k(x, \la)$ is a non-trivial issue (see the discussion in \cite[Section~4]{Bond22}). For the coefficients $\tau_1 \in L_2(0,1)$ and $\tau_0 \in W_2^{-1}(0,1)$ of the differential expression
\begin{equation} \label{pq}
y''' + (\tau_1 y)' + \tau_1 y' + \tau_0 y,
\end{equation}
the reconstruction formulas have been obtained in \cite[Theorem~3]{Bond22}.

For $s = 1$, we deduce from \cite[Theorem~3]{Bond22} the following formula for recovering $\sigma(x)$:
\begin{equation} \label{ser1}
\sigma = \tilde \sigma - \frac{3}{2s} \sum_{(l,k,\eps) \in V} \bigl( \phi_{l,k,\eps}(x) \tilde \eta_{l,k,\eps}(x) - A_{l,k,\eps}\bigr),
\end{equation}
where the notations of \cite{Bond22} are used. The asymptotics of spectral data can be obtained from \cite{Bond22-jms}. In particular, the eigenvalues of the boundary value problems \eqref{sys}, \eqref{bcY} have the asymptotics
\begin{equation} \label{asymptla}
\la_{n,k} = (-1)^{k-1} \left( \frac{2\pi}{\sqrt 3} \bigl(n + \chi_k + \varkappa_{n,k}\bigr)\right)^3, \quad n \in \mathbb N, \quad k = 1, 2,
\end{equation}
where $\{ \varkappa_{n,k} \}$ is a sequence of the Fourier coefficients of an $L_3$-function. Such asymptotics imply the convergence of the series \eqref{ser1} in $L_3(0,1)$ for an appropriate choice of regularizing constants $A_{l,k,\eps}$ due to \cite[Lemma~8]{Bond22}. Furthermore, the series \eqref{ser1} can be applied to obtain stability estimates for the inverse problem solution.

Anyway, the study of the differential expression \eqref{expr} and of the related inverse spectral problems implies several open problems:

\smallskip

\textit{Regularization}. It is unknown how to construct an associated matrix for the differential expression \eqref{pq}
with two independent coefficients $\tau_1 \in W^{-1}_{3,loc}(0,T)$ and $\tau_0 \in W_{3,loc}^{-2}(0,T)$.
Also, generalizations to higher odd orders should be studied.

\smallskip

\textit{Reconstruction for $s = 0$}. In the case $s = 0$, Theorem~3 in \cite{Bond22} implies the following formal series for $\sigma(x)$:
\begin{equation} \label{ser}
\sigma'(x) = \tilde \sigma'(x) + 3 \kappa^{-1} \sum_{(l,k,\eps) \in V} \phi_{l,k,\eps}(x) \tilde \eta'_{l,k,\eps}(x),
\end{equation}
where the notations of \cite{Bond22} are used. The convergence of this series in $W_3^{-1}(0,1)$ for $\sigma \in L_3(0,1)$ is a challenge. The asymptotics of form \eqref{asymptla} do not directly imply the convergence of the series \eqref{ser}, as in \cite{Bond22, Bond23-res}.

\smallskip

\textit{Existence}. A small perturbation of the spectral data for the differential expression \eqref{expr}
leads to the independence of two coefficients as in \eqref{pq}. This is an obstacle in obtaining local and global solvability of inverse problems for the system \eqref{sys} with the matrix \eqref{defF}.

\medskip

{\bf Funding.} This work was supported by Grant 24-71-10003 of the Russian Science Foundation, https://rscf.ru/en/project/24-71-10003/.

\smallskip

{\bf Acknowledgment.} The author is grateful to Prof. Sergey Buterin for his valuable comments.

\noindent Natalia Pavlovna Bondarenko \\

\noindent 1. Department of Mechanics and Mathematics, Saratov State University, 
Astrakhanskaya 83, Saratov 410012, Russia, \\

\noindent 2. Department of Applied Mathematics, Samara National Research University, \\
Moskovskoye Shosse 34, Samara 443086, Russia, \\

\noindent 3. S.M. Nikolskii Mathematical Institute, RUDN University, 6 Miklukho-Maklaya St, Moscow, 117198, Russia, \\

\noindent 4. Moscow Center of Fundamental and Applied Mathematics, Lomonosov Moscow State University, Moscow 119991, Russia.\\

\noindent e-mail: {\it bondarenkonp@sgu.ru}

\end{document}